\documentclass[english]{llncs}

\makeatletter

\usepackage{lipsum}
\usepackage{amsfonts}
\usepackage{graphicx}
\usepackage{epstopdf}
\usepackage{algorithmic,nicematrix,colortbl}
\usepackage{amsmath, amssymb, xcolor, xfrac, amsfonts,tikz}
\usepackage{float, hyperref}
\usepackage[inline]{enumitem}

\ifpdf
  \DeclareGraphicsExtensions{.eps,.pdf,.png,.jpg}
\else
  \DeclareGraphicsExtensions{.eps}
\fi

\newcommand{\mbf}[1]{\mathbf{#1}}
\newcommand{\mbb}[1]{\mathbb{#1}}

\newcommand{\mcf}[1]{\mathcal{#1}}
\newcommand{\Mod}[1]{\ \mathrm{mod}\ #1}

\newenvironment{cpf}{\begin{trivlist} \item[] {\em Proof of Claim.}}{\hspace*{\stretch{1}} $\diamond$ \end{trivlist}}

\DeclareMathOperator{\Span}{span}

\DeclareMathOperator{\oct}{oct}
\DeclareMathOperator{\ocp}{ocp}

\DeclareMathOperator{\STAB}{STAB}

\newcommand{\rank}{\operatorname{rank}}
\DeclareMathOperator{\rec}{rec}

\usepackage{amsopn}

\DeclareMathOperator{\conv}{conv}

\title{Extended formulations for the integer hull of strictly \texorpdfstring{$\Delta$}{Delta}-modular cographic polyhedral cones}

\author{Joseph Paat\inst{1}
\and Zach Walsh \inst{2}
\and Luze Xu\inst{3}}

\institute{Sauder School of Business, University of British Columbia, Canada\\
\email{joseph.paat@sauder.ubc.ca}\\
\and Department of Mathematics and Statistics, Auburn University, USA\\
\email{zwalsh@auburn.edu}\\
\and Department of Mathematics, University of California Davis, USA\\
\email{lzxu@ucdavis.edu}
}

\begin{document}

\maketitle
\thispagestyle{plain}

\begin{abstract}
Conforti et al. give a compact extended formulation for a class of bimodular-constrained integer programs, namely those that model the stable set polytope of a graph with no disjoint odd cycles.
We extend their techniques to design compact extended formulations for the integer hull of translated polyhedral cones whose constraint matrix is strictly $\Delta$-modular and has rows that represent a cographic matroid.
Our work generalizes the important special case from Conforti et al. concerning $4$-connected graphs with odd cycle transversal number at least $4$.
We also discuss the necessity of our assumptions.
\end{abstract}

\section{Introduction}
%

For a constraint matrix $\mbf{A} \in \mbb{Z}^{m\times n}$ with $\rank(\mbf{A}) = n$ and a right-hand side $\mbf{b} \in \mbb{Z}^m$, define the polyhedron 
\[
\mcf{P}(\mbf{A}, \mbf{b}) := \left\{\mbf{x} \in \mbb{R}^n: \mbf{A} \mbf{x} \le \mbf{b}\right\}.
\]
The integer programming problem (IP) is to check whether $\mcf{P}(\mbf{A}, \mbf{b}) \cap \mbb{Z}^n$ is non-empty. 
It remains an open conjecture if (IP) can be solved in polynomial time if the largest $n\times n$ sub-determinant of $\mbf{A}$, denoted by
\[
\Delta(\mbf{A}) := \max\left\{|\det\mbf{B}|:\ \mbf{B}~\text{is an $n\times n$ submatrix of $\mbf{A}$}\right\},
\]
is fixed.
We say $\mbf{A}$ is {\bf $\Delta$-modular} if $\Delta(\mbf{A}) \le \Delta$.

A collection of work supports the conjecture that (IP) can be solved efficiently when $\mbf{A}$ is $\Delta$-modular.
If $\Delta(\mbf{A}) = 1$, then (IP) can be solved by linear programming because $\mbf{A}$ is equivalent (up to a unimodular mapping) to a {\bf totally unimodular (TU)} matrix; see, e.g.,~\cite{AS1986}.
For $\Delta(\mbf{A}) \le 2$, where $\mbf{A}$ is {\bf bimodular}, Artmann et al.~\cite{AWZ2017} show that (IP) can be solved in strongly polynomial time. 
For $\Delta \ge 3$, the conjecture has been verified under extra assumptions, see, e.g.,~\cite{AFJKSWY2024,AEGOVW2016,NaSaZe2022,NNSZ2023,FiJoWeYu2022}.
There are also many structural results to better understand $\Delta$-modular matrices; see, e.g.,~\cite{AS2022,GWZ2018,gribanov2022delta,OW2021,PSWX2023}.
However, to the best of our knowledge, the conjecture remains open for $\Delta \ge 3$.

The bimodular algorithm by Artmann et al.~\cite{AWZ2017} uses Seymour's decomposition of TU matrices; see~\cite{S1980}.
In a different approach, Conforti et al.~\cite{CFHW2022} design polynomial-size extended formulations for a class of bimodular problems.
An {\bf extended formulation} of a polyhedron $\mcf{E} \subseteq \mbb{R}^n$ is a description
\[
\mcf{E} = \left\{\mbf{x} \in\mbb{R}^n:\ \exists~\text{$\mbf{y} \in \mbb{R}^k$ such that}\ \mbf{B}\mbf{x}+\mbf{C}\mbf{y} \le \mbf{d}\right\},
\]
where $\mbf{B} \in \mbb{R}^{\ell \times n}$, $\mbf{C} \in \mbb{R}^{\ell \times k}$, and $\mbf{d}\in \mbb{R}^{\ell}$.
The number $\ell$ is the {\bf size} of the extended formulation. 
One important aspect of extended formulations is the following: if $\conv(\mcf{P}(\mbf{A}, \mbf{b}) \cap \mbb{Z}^n)$ admits an extended formulation that can be constructed in polynomial time, then the corresponding (IP) can be solved in polynomial time by linear programming.

Conforti et al.~\cite{CFHW2022} consider the stable set polytope of a graph $G$:
\begin{align*}
\STAB(G)&:=\conv\{\mbf{x}\in\mbb{Z}_+^n : \mbf{A}\mbf{x}\le \mbf{1}\}\\
&=\conv(\mcf{P}\left(\mbf{A}', \mbf{b}')\cap\mbb{Z}^n\right),
\end{align*}
where $\mbf{A}$ is the edge-node incidence matrix of $G$, and
\[
\mbf{A}'=\begin{bmatrix}\mbf{A}\\-\mbf{I}_n\end{bmatrix}\quad\text{and}\quad\mbf{b}'=\begin{bmatrix}\mbf{1}\\\mbf{0}\end{bmatrix}.
\]
It is known that $\mbf{A}'$ is bimodular if $\ocp(G) = 1$, where $\ocp(G)$ is the maximum number of node-disjoint odd cycles in $G$; see~\cite{GKS1995}.
Conforti et al.~\cite{CFHW2022} show that the stable set polytope $\STAB(G)$ with $\ocp(G) = 1$ admits an extended formulation of size in $O(n^2)$ that can be constructed in polynomial time.

In this note, we look to expand the family of polyhedra $\mcf{P}(\mbf{A}, \mbf{b})$ for which we can construct polynomial-sized extended formulations for $\conv(\mcf{P}(\mbf{A}, \mbf{b}) \cap \mbb{Z}^n)$.
In particular, we consider polyhedra defined by {\bf strictly $\Delta$-modular matrices}, which are matrices $\mbf{A}$ such that all $n\times n$ submatrices $\mbf{B}$ of $\mbf{A}$ satisfy $\det\mbf{B} \in \{0, \pm \Delta\}$.
For strictly $\Delta$-modular matrices, the bounded determinant IP conjecture has been answered for $\Delta \le 4$; see~\cite{NNSZ2023,NaSaZe2022}. 
Interestingly, the conjecture seems to be open for strictly $\Delta$-modular matrices with $\Delta \ge 5$.

\section{Results}

Our main result covers integer programs defined by a translated cone $\mcf{P}(\mbf{A}, \mbf{b})$ and a strictly $\Delta$-modular matrix $\mbf{A}$ such that the column matroid $M(\mbf{A}^\top)$ of $\mbf{A}^\top$, i.e., the row matroid of $\mbf{A}$, is {\bf cographic}. 
Note $\mcf{P}(\mbf{A}, \mbf{b})$ is a translated cone if and only if $\mbf{b}$ is in the column span of $\mbf{A}$, which we denote by $\Span(\mbf{A})$.
\begin{theorem}\label{thmMain}
Fix $\Delta \in \mbb{N}$.
Let $\mbf{A} \in \mbb{Z}^{m\times n}$ have $\rank(\mbf{A}) = n$ and $\mbf{b} \in \mbb{Z}^m$.
If
\begin{enumerate}[label=(\roman*), leftmargin = *]
    \item\label{Assume1}$\mbf{A}$ is strictly $\Delta$-modular,
    \item\label{Assume2}the column matroid $M(\mbf{A}^\top)$ is cographic, and
    \item\label{Assume3}$\mbf{b} \in \mathrm{span}(\mbf{A})$,
\end{enumerate}
then $\conv(\mcf{P}(\mbf{A}, \mbf{b}) \cap\mbb{Z}^n)$ admits an extended formulation of size in $O(n^{\Delta})$ that can be constructed in polynomial time.
\end{theorem}

Theorem~\ref{thmMain} and the underlying analysis come with a handful of consequences. 
Before discussing these consequences, we mention that Aprile et al.~\cite{AFJKSWY2024} refer to the cographic case as that in which $\mbf{A}$ is the transpose of a network matrix, so $M(\mbf{A}^\top)$ is a graphic matroid.
In contrast, we assume $M(\mbf{A}^\top)$ is cographic.
We also mention that Condition \emph{\ref{Assume3}} is inspired by the work of Conforti et al.~\cite{CFHW2022} and of Veselov and Chirkov~\cite{VC2009}, who demonstrate that if $\mbf{A}$ is bimodular, then $\conv(\mathcal{P}\cap\mbb{Z}^n)$ can be characterized by the integer hulls of the vertex cones.

In addition to new families of extended formulations, we believe that one of the most interesting contributions is a technical lemma that may be of independent interest; see Lemma~\ref{lem:compact_ef_delta}.
We briefly outline our proof of Theorem~\ref{thmMain}, and consequently, we contextualize Lemma~\ref{lem:compact_ef_delta}.
First, we use Lemmas~\ref{lem:PQ_U} and~\ref{lem:reformulation_delta} to reformulate $\conv(\mcf{P}(\mbf{A}, \mbf{b}) \cap\mbb{Z}^n)$ into a polyhedron $\mcf{Q}(\mbf{A}, \mbf{b})$ in the space of slack variables.
Next, using Conditions~\emph{\ref{Assume1}}-\emph{\ref{Assume3}}, we argue that $\mcf{Q}(\mbf{A}, \mbf{b})$ can be viewed as a congruency-constrained uncapacitated circulation problem.
In this circulation problem, the congruency constraint depends on $\Delta$.
Hence, we can use our bound on $\Delta$ from Condition \emph{\ref{Assume1}} to construct a compact extended formulation; this is Lemma~\ref{lem:compact_ef_delta}.
Conforti et al.~\cite{CFHW2022} use a similar transformation, and some of their arguments appear in a more general form in Conforti et al.~\cite{CFHSW2020}.
However, our work provides an alternative proof that allows us to generalize parts of their work. 
%
%
We collect lemmas to prove Theorem~\ref{thmMain} in Section~\ref{prelimResults} and prove Theorem~\ref{thmMain} in Section~\ref{secProof}.

Our analysis also allows us to relax the `strictly' part of Condition~\emph{\ref{Assume1}} to $\Delta$-modular when $\gcd(\mbf{A})=1$, where
\[
\gcd(\mbf{A}) := \gcd\{|\det \mbf{B}|:\mbf{B}~\text{is a $n \times n$ submatrix of}~\mbf{A}\}.
\]
If $\gcd(\mbf{A})=1$, then $\mcf{P}(\mbf{A}, \mbf{b}) = \conv(\mcf{P}(\mbf{A}, \mbf{b}) \cap \mbb{Z}^n)$; see Remark~\ref{remark2}.
In particular, the conclusion of Theorem~\ref{thmMain} holds if $\mbf{A}$ is bimodular and satisfies Conditions~\emph{\ref{Assume2}} and~\emph{\ref{Assume3}}.

As a final consequence, we can use Theorem~\ref{thmMain} and a structural lemma (Lemma 8 in \cite{CFHW2022}) to recover one of the results by Conforti et al.
To state their result, define the {\bf odd cycle transversal number} $\oct(G)$ of a graph $G = (V,E) $ as the minimum cardinality of a set of nodes $X \subseteq V$ such that $G - X$ is bipartite.

\begin{corollary}\label{corConforti}({\bf Follows from Theorem 2 and Theorem 14 in \cite{CFHW2022}})
    If $G$ is a $4$-connected graph on $n$ nodes with $\ocp(G) \le 1$ and $\oct(G) \ge 4$, then $\STAB(G)$ admits an extended formulation whose size is in $O(n^2)$.
\end{corollary}

\noindent 

In their proof of Corollary~\ref{corConforti}, Conforti et al. use a characterization due to Lov\'{a}sz stating that a $4$-connected graph $G$ with $\ocp(G) \le 1$ and $\oct(G) \ge 4$ has an even-face embedding in the projective plane.
Our approach offers a direct algebraic proof to sidestep the complexities introduced by the even-face embedding argument.
Our proof of Corollary~\ref{corConforti} follows from Theorem \ref{thmMain} via a result of Lov\'{a}sz (see Theorem \ref{thm:Lovasz}) that the column matroid of the node-edge incidence matrix of $G$ is cographic when $G$ is $4$-connected with $\ocp(G) \le 1$ and $\oct(G) \ge 4$; see Section~\ref{secGeneralizing}.

Theorem \ref{thmMain} extends Corollary \ref{corConforti} in two directions.
First, the results in \cite{CFHW2022} apply only with $\mbf b = \mbf 1$, while Theorem \ref{thmMain} holds for any $\mbf{b}\in \Span(\mbf A)$.
Second, the corollary applies only to a specific class of cographic bimodular matrices (i.e., the matrix $\mbf{A}'$ defining $\STAB(G)$), while Theorem \ref{thmMain} holds for all bimodular cographic matrices; see our extension from strictly bimodular to bimodular in Remark~\ref{remark2}.
In particular, it follows from a result of Slilaty~\cite{SLILATY2005207} that if $\mbf {A}^\top$ represents a cographic matroid of a graph $H$ and is also the node-edge incidence matrix of a graph, then $H$ is projective-planar.
Therefore, the following corollary of Theorem \ref{thmMain} gives a new family of matrices for which the corresponding polytope has a polynomial-size extended formulation; these results do not follow from work in \cite{CFHW2022}.
We discuss this in Section~\ref{secNew}.

\begin{corollary}\label{corNewFamilies}
    Let $H$ be a graph that is not projective-planar, $\mbf{A}^\top$ be bimodular with $\rank(\mbf{A}^\top) = n$ that represents the cographic matroid of $H$, and $\mbf b \in \mathrm{span}(\mbf A)$.
    It holds that $\conv(\mcf{P}(\mbf{A}, \mbf{b}) \cap\mbb{Z}^n)$ admits an extended formulation of size in $O(n^2)$.
\end{corollary}

Consider the following setting covered by Corollary~\ref{corNewFamilies}.
If $H$ is a non-projective-planar graph, then the cographic matroid of $H$ has a representation as the row matroid of a unimodular matrix $\mbf U$.
We can then apply any bimodular column operations and row scalings to obtain a bimodular matrix $\mbf A$ so that $M(\mbf{U}^\top) = M(\mbf{A}^\top)$.
In particular, $M(\mbf{A}^\top)$ is cographic, and Corollary \ref{corNewFamilies} applies to $\mbf{A}^\top$.

A natural question is how to further generalize Theorem~\ref{thmMain}.
It turns out that neither \emph{\ref{Assume2}} nor \emph{\ref{Assume3}} can be dropped for the statement of Theorem~\ref{thmMain} to remain true.
In order to sees this, we turn to examples by Cevallos et al.~\cite{CWZ2018} and Jia et al.~\cite{jia2023exact} in Section~\ref{secExamples}.

\section{Preliminaries and notation}\label{secPreliminaries}

We use upper case bold font to denote matrices, and lower case bold font for vectors.
For $\mbf{x} \in \mbb{R}^n$, we denote the $i$th entry of $\mbf{x}$ by $x_i$.
For $I \subseteq \{1, \dotsc, n\}$, we set $\mbf{x}(I) := \sum_{i \in C} x_i$.
We use $\mbf{1}$ and $\mbf{0}$ to denote the matrix (or vector) of all-ones and all-zeros, respectively; the dimensions are implied by context.
We use $\mbf{I}_r$ to denote the $r\times r$ identity matrix.
The $i$th standard unit vector in $\mbb{R}^n$ is denoted by $\mbf{e}_i$.
We use $\mbb{Z}_+$ and $\mbb{R}_+$ to denote the set of non-negative integers and real numbers, respectively.
For an invertible matrix $\mbf{B} \in \mbb{Z}^{n\times n}$ and vectors $\mbf{u}, \mbf{v} \in \mbb{Z}^n$, we write $\mbf{u} \equiv \mbf{v} \Mod \mbf{B}\mbb{Z}^n$ if $\mbf{u} - \mbf{v} \in \mbf{B}\mbb{Z}^n$.
The relation$\Mod \mbf{B}\mbb{Z}^n$ partitions $\mbb{Z}^n$ into $|\det \mbf{B}|$ equivalence classes; see, e.g.,~\cite{AS1986}.

We use standard notation in matroid theory; for more, see~\cite{Oxley}.
A {\bf matroid} is a pair $M = (E,r)$, where $E = E(M)$ is a finite ground set and $r:2^E \to \mbb{Z}$ is a non-decreasing submodular rank function.
We write matroids, their ground sets, and elements within using italic font.
If $\mbf{A}$ is a matrix over a field $\mbb{F}$ with columns indexed by $E$, then the function $r : 2^E \to \mathbb Z$ defined by $r(X) := \rank(\mbf{A}[X])$ defines a matroid $M(\mbf A)$ on $E$ called the {\bf column matroid} of $\mbf A$, where $\mbf{A}[X]$ is the set of columns of $\mbf{A}$ corresponding to $X$.
We say that $M(\mbf A)$ is {\bf $\mathbb F$-representable}.
If $M = M(\mbf A)$ for an integer matrix $\mbf A$ (viewed over $\mathbb R$ or $\mathbb Q$) with $\Delta(\mbf A) = 1$, then $M$ is {\bf regular}.
A {\bf signed node-edge incidence matrix} of a graph $H$ is the node-arc incidence matrix of a directed version of $H$; the column corresponding to edge $\{u,v\}$ is the difference of the unit vectors $\mbf e_u$ and $\mbf e_v$.
 If $\mbf A$ is a signed node-edge incidence matrix of a graph $H$, then $M(\mbf A)$ is also the {\bf graphic matroid} of $H$, denoted by $M(H)$.
The rank function of $M(H)$ is $r(X) = |V(H[X])| - c(H[X])$ for all sets $X$ of edges of $H$, where $c(H[X])$ is the number of connected components of the subgraph $H[X]$.
Let $K_r$ be the complete graph on $r$ nodes.
The canonical representation of $M(K_r)$ over any field is 
\[
[\begin{array}{@{\hskip .1 cm}c@{\hskip .15 cm}c@{\hskip .1 cm}}\mbf{I}_{r-1}&\mbf{D}_{r-1}\end{array}] \in \mbb{Z}^{(r-1)\times ((r-1)+\binom{r-1}{2})},
\]
where the columns of $\mbf{D}_{r-1}$ are $\mbf{e}_i - \mbf{e}_j$ for all $1 \le i < j \le r-1$.
The {\bf dual} of a matroid $M = (E, r)$ is the matroid $M^* = (E, r^*)$, where $r^*(X) = |X| - (r(E) - r(E - X))$ for all $X \subseteq E$. 
If $M$ is the dual of a graphic matroid, then $M$ is {\bf cographic}. The canonical representation of $M(K_r)^*$ over any field is $
[-\mbf{D}_{r-1}^\top ~\mbf{I}_{\binom{r-1}{2}}].$

\section{Lemmas for Theorem~\ref{thmMain}.}\label{prelimResults}

%
The first setting we handle in Theorem~\ref{thmMain} is when $\mbf{b}$ is an integer linear combination of the columns of $\mbf{A}$.
Here, $\mcf{P}(\mbf{A}, \mbf{b})$ has an integral vertex, so it is already a compact extended formulation of $\conv(\mcf{P}(\mbf{A}, \mbf{b})\cap \mbb{Z}^n)$.
%
%

\begin{lemma}\label{lem:rational_cone}
Let $\mbf{A} \in \mbb{Z}^{m\times n}$ be $\Delta$-modular with $\rank(\mbf{A}) = n$ and $\mbf{b} \in \mbb{Z}^m$.
If $\mbf{A}\overline{\mbf{x}} = \mbf{b}$ for some $\overline{\mbf{x}}\in\mbb{Z}^n$, then $\mcf{P}(\mbf{A}, \mbf{b}) = \conv(\mcf{P}(\mbf{A}, \mbf{b})\cap \mbb{Z}^n)$ and has $O(n^2\Delta^2)$ many facets.
\end{lemma}
\begin{proof}
As $\rank(\mbf{A}) = n$ and $\mbf{A}\overline{\mbf{x}}= \mbf{b}$, we have that $\mcf{P}(\mbf{A}, \mbf{b})$ is a translated rational cone with vertex $\overline{\mbf{x}} \in \mbb{Z}^n$.
The integer hull of a rational cone is the cone itself, so $\mcf{P}(\mbf{A}, \mbf{b}) = \conv(\mcf{P}(\mbf{A}, \mbf{b})\cap \mbb{Z}^n)$.
The facets of $\mcf{P}(\mbf{A}, \mbf{b})$ correspond to distinct rows of $\mbf{A}$, of which there are at most $O(n^2\Delta^2)$; see Theorem 2 in~\cite{LPSX2022}.
\qed
\end{proof}

The remaining lemmas address the setting in which $\mbf{b}$ is not an integer linear combination of the columns of $\mbf{A}$.
As in~\cite{CFHW2022}, we perform an affine transformation of $\conv(\mcf{P}(\mbf{A}, \mbf{b}) \cap \mbb{Z}^n)$ into the space of slack variables.
For a matrix $\mbf{A} \in \mbb{Z}^{m\times n}$ with $\rank(\mbf{A}) = n$ and $\mbf{b} \in \mbb{Z}^m$, define
\begin{equation}\label{eqQ}
\mcf{Q}(\mbf{A}, \mbf{b}) := \conv\left\{\mbf{b}-\mbf{A}\mbf{x}\in\mbb{Z}_{+}^m:~\mbf{x}\in\mbb{Z}^n\right\}.
\end{equation}

\begin{lemma}\label{lem:PQ_U}
Let $\mbf{A} \in \mbb{Z}^{m\times n}$ with $\rank(\mbf{A}) = n$ and $\mbf{b} \in \mbb{Z}^m$.
It holds that
\[
\mcf{Q}(\mbf{A}, \mbf{b}) = \left\{\mbf{b}-\mbf{A}\mbf{x}: \mbf{x}\in\conv(\mcf{P}(\mbf{A}, \mbf{b}) \cap \mbb{Z}^n)\right\}.
\]
%
%
\end{lemma}
\begin{proof}
    The 
    equation follows from the definitions of $\mcf{Q}(\mbf{A}, \mbf{b})$ and $\conv(\mcf{P}(\mbf{A}, \mbf{b}) \cap \mbb{Z}^n)$, and because the convex hull operation commutes with affine maps.
    %
    %
\qed
\end{proof}
Next, we reformulate $\mcf{Q}(\mbf{A}, \mbf{b})$ as a polyhedron with extra congruency constraints.
In order to do this, we consider $\mbf{U}\in\mbb{Z}^{n\times m}$, $\mbf{R}\in\mbb{Z}^{(m-n)\times m}$ and $\mbf{H} \in \mbb{Z}^{n\times n}$ satisfying $|\det \mbf{H}| = \gcd(\mbf{A})$ and
\begin{equation}\label{eqHNF}
\begin{bmatrix}
    \mbf{U}\\
    \mbf{R}
\end{bmatrix}~\text{is unimodular with}~
\begin{bmatrix}
    \mbf{U}\\
    \mbf{R}
\end{bmatrix}\mbf{A} =
\begin{bmatrix}
    \mbf{H}\\
    \mbf{0}
\end{bmatrix}.
\end{equation}
The matrices $\mbf{U}$, $\mbf{R}$, and $\mbf{H}$ can be found, e.g., by transforming $\mbf{A}$ into Hermite Normal Form; see, e.g.,~\cite[page 48]{AS1986}.

\begin{lemma}\label{lem:reformulation_delta}
Let $\mbf{A} \in \mbb{Z}^{m\times n}$ with $\rank(\mbf{A}) = n$ and $\mbf{b} \in \mbb{Z}^m$.
Let $\mbf{U}\in\mbb{Z}^{n\times m}$, $\mbf{R}\in\mbb{Z}^{(m-n)\times m}$ and $\mbf{H} \in \mbb{Z}^{n\times n}$ satisfy~\eqref{eqHNF}.
It holds that $\mcf{Q}(\mbf{A}, \mbf{b}) = \conv \mcf{R}(\mbf{A}, \mbf{b})$, where
\begin{equation}\label{eqQR_delta}
\mcf{R}(\mbf{A}, \mbf{b}) := \left\{\mbf{y}\in\mbb{Z}_{+}^m:
\begin{array}{@{\hskip 0 cm}r@{\hskip .1cm}c@{\hskip .1cm}l}
\mbf{R}(\mbf{y} -\mbf{b}) &=& \mbf{0}~\text{and}\\
\mbf{U}(\mbf{y}-\mbf{b})&\equiv& \mbf{0} \Mod \mbf{H}\mbb{Z}^n
\end{array}\right\}.
\end{equation}
\end{lemma}
\begin{proof}
To show $\mcf{Q}(\mbf{A}, \mbf{b})\subseteq\conv\mcf{R}(\mbf{A}, \mbf{b})$, let $\mbf{y}:= \mbf{b} - \mbf{A}\mbf{x} \in \mcf{Q}(\mbf{A}, \mbf{b})$, where $\mbf{x} \in \mcf{P}(\mbf{A}, \mbf{b}) \cap \mbb{Z}^n$. 
Note $\mbf{y} \in \mbb{Z}^m_+$.
We have 
\(
\mbf{R}(\mbf{y} -\mbf{b})=\mbf{R}(-\mbf{A}\mbf{x})=\mbf{0}
\)
and
\[
\mbf{U}(\mbf{y}-\mbf{b}) = \mbf{U}(-\mbf{A}\mbf{x})=-\mbf{H}\mbf{x}\in \mbf{H}\mbb{Z}^n.
\]
Therefore, $\mbf{y}\in\mcf{R}(\mbf{A}, \mbf{b})$, and $\mcf{Q}(\mbf{A}, \mbf{b}) \subseteq \conv\mcf{R}(\mbf{A}, \mbf{b})$.

To show $\conv\mcf{R}(\mbf{A}, \mbf{b})\subseteq\mcf{Q}(\mbf{A}, \mbf{b})$, let $\mbf{y}\in\mcf{R}(\mbf{A}, \mbf{b})$.
Set $\mbf{x}:=\mbf{H}^{-1}\mbf{U}(\mbf{b}-\mbf{y})$. 
We have $\mbf{x}\in\mbb{Z}^n$ because $\mbf{U}(\mbf{y}-\mbf{b})\equiv \mbf{0} \Mod \mbf{H}\mbb{Z}^n$.
Also, 
\[
\mbf{b} - \mbf{A}\mbf{x} = \mbf{b} - \mbf{A}\mbf{H}^{-1}\mbf{U}(\mbf{b}-\mbf{y}) = \mbf{b} - \begin{bmatrix}
    \mbf{U}\\
    \mbf{R}
\end{bmatrix}^{-1}
\begin{bmatrix}
    \mbf{U}(\mbf{b}-\mbf{y})\\
    \mbf{0}
\end{bmatrix}.
\]
Given that $\mbf{R}(\mbf{b}-\mbf{y})=\mbf{0}$, we have 
\[
\mbf{b} - \mbf{A}\mbf{x} = \mbf{b} - \begin{bmatrix}
    \mbf{U}\\
    \mbf{R}
\end{bmatrix}^{-1}
\begin{bmatrix}
    \mbf{U}\\
    \mbf{R}
\end{bmatrix}(\mbf{b}-\mbf{y})=\mbf{y}\ge \mbf{0}.
\]
Therefore, $\conv\mcf{R}(\mbf{A}, \mbf{b})\subseteq\mcf{Q}(\mbf{A}, \mbf{b})$.
\qed
\end{proof}
%


%
\begin{remark}\label{remark2}
Lemma~\ref{lem:reformulation_delta} does not require $\mbf{A}$ to be strictly $\Delta$-modular nor $\mbf{b} \in \Span(\mbf{A})$.
However, we can rewrite the conclusion with these assumptions.
If $\mbf{A}$ is strictly $\Delta$-modular, then $\gcd(\mbf{A})=\Delta$ and we can choose $\mbf{H}$ to be a submatrix of $\mbf{A}$.
If $\gcd(\mbf{A})=1$, then we can choose $\mbf{H}=\mbf{I}_n$.
Moreover, if $\gcd(\mbf{A})=1$ and $\mbf{b}\in\mathrm{span}(\mbf{A})$, then $\mcf{Q}(\mbf{A}, \mbf{b}) $ equals
\[
\conv\left\{\mbf{y}\in\mbb{Z}^m_+: \mbf{R}\mbf{y}=\mbf{0}\right\} = \left\{\mbf{y}\in\mbb{R}^m_+: \mbf{R}\mbf{y}=\mbf{0}\right\},
\]
where the equation holds because the right hand side is a rational cone.
Note that every bimodular matrix $\mbf{A}$ is either strictly bimodular or satisfies $\gcd(\mbf{A})=1$.
\hfill $\diamond$
\end{remark}
%

The following results can be found in~\cite{Oxley}.
For $\mbf{X} \in \mbb{R}^{p \times q}$, let $\mbf{X}^\# \in \{0,1\}^{p\times q}$ denote the matrix obtained by replacing every non-zero entry in $\mbf{X}$ by a $1$.
Lemmas~\ref{lemOx641} and~\ref{lemOx1017} are used later to prove that a cographic matroid with a strictly $\Delta$-modular matrix representation has an equivalent node-arc incidence matrix representation up to row operations.

\begin{lemma}[Proposition 6.4.1 in~\cite{Oxley}]\label{lemOx641}
Let $[\mbf{I}_r~\mbf{D}_1]$ and $[\mbf{I}_r~\mbf{D}_2]$ be matrices over fields $\mbb{F}_1$ and $\mbb{F}_2$, respectively, with the columns of each labelled, in order, by $\mbf{f}_1, \dotsc, \mbf{f}_t$.
If the identity map on $\{\mbf{f}_1, \dotsc, \mbf{f}_t\}$ is an isomorphism from $M([\mbf{I}_r~\mbf{D}_1])$ to $M([\mbf{I}_r~\mbf{D}_2])$, then $\mbf{D}_1^\# = \mbf{D}_2^\#$.
\end{lemma}

The following lemma originates from~\cite{C1963}.

\begin{lemma}[Lemma 10.1.7 in~\cite{Oxley}]\label{lemOx1017}
Let $\mbf{D}_1$ and $\mbf{D}_2$ be TU matrices. 
If $\mbf{D}_1^{\#} = \mbf{D}_2^{\#}$, then $\mbf{D}_2$ can be obtained from $\mbf{D}_1$ by multiplying some rows and columns by $-1$.
\end{lemma}

Lemma 13 in~\cite{CFHW2022} gives a compact extended formulation for the uncapacitated odd parity constrained circulation problem, which is used to give a compact extended formulation for certain bimodular-constrained IPs related to the stable set polytope.
We give a generalized result to hold for more general congruency constraints.

\begin{lemma}\label{lem:compact_ef_delta}
Let $\mbf{D}$ be the node-arc incidence matrix of a directed graph $H = (V,A)$, let $d \in \mbb{Z}_+$ and $\mbf{H} \in \mbb{Z}^{d\times d}$ with $\Delta = |\det \mbf{H}|$, let $\mbf{W} \in \mbb{Z}^{d\times A}$, and let $\mbf{f} \in \mbb{Z}^d$.
Consider
\[
\mcf{S} := \conv\left\{\mbf{y}\in\mbb{Z}_{+}^A: \mbf{D}\mbf{y}= \mbf{0}~\text{and}~\mbf{W}\mbf{y}\equiv \mbf{f}\Mod \mbf{H} \mbb{Z}^{d}\right\},
\]
which is the convex hull of nonnegative integer circulations in $H$ with a congruency constraint over $\mbf{H}\mbb{Z}^d$.
Then $\mcf{S}$ admits an extended formulation of size in $O(\Delta^{\Delta+1}|A|\cdot |V|^{\Delta-1})$ that can be constructed in polynomial time.
\end{lemma}
\begin{proof}
There are vectors $\mbf{g}^0 := \mbf{0},\mbf{g}^1, \dotsc, \mbf{g}^{\Delta-1}$ in $\mbb{Z}^d \cap \mbf{H} [0,1)^d$ that represent the cosets of $\mbb{Z}^d / \mbf{H} \mbb{Z}^d$; see, e.g.,~\cite{AS1986}.
Construct the directed graph $H' = (V', A')$ with
%
\[
\begin{array}{@{\hskip .1 cm}r@{\hskip .1 cm}c@{\hskip .1 cm}l}
V' &:= & \bigcup_{i=0}^{\Delta-1} \left\{(v,\mbf{g}^i): v\in V\right\} \quad\text{and}\\[.25cm]
A'&:= & \left\{\left((v, \mbf{g}^i),(w, \mbf{g}^j)\right) : \begin{array}{@{\hskip 0 cm}l@{\hskip 0 cm}}
(v,w)\in A,\\
\mbf{g}^j  \equiv (\mbf{g}^i + \mbf{W}_{(v,w)}) \Mod \mbf{H}\mbb{Z}^d
\end{array}\right\},
\end{array}
\]
where $\mbf{W}_{(v,w)}$ is the column of $\mbf{W}$ corresponding to $(v,w) \in A$.
Note that $|V'| = \Delta \cdot |V|$ and $|A'| = \Delta \cdot |A|$ because, for each $(v,w)\in A$ and $i\in\{0,1,\dots,\Delta-1\}$, there is a unique $\mbf{g}^j$ congruent to $\mbf{g}^i + \mbf{W}_{(v,w)}$.
Let $\mbf{D}'$ be the node-arc incidence matrix of $H'$.

Let $\Omega$ denote the set of subsequences of $(\mbf{g}^i)_{i=1}^{\Delta-1}$ whose sum is congruent to $\mbf{f}$ and with no zero-sum subsequence:
\[
\Omega := \left\{(\mbf{g}^{t_i})_{i=1}^k:
\begin{array}{@{\hskip 0 cm}l@{\hskip .1 cm}}
k \ge 1, t_1, \dotsc, t_k \in \{1, \dotsc, \Delta-1\},\\
\sum_{i =1 }^k\mbf{g}^{t_i} \equiv \mbf{f} \Mod \mbf{H}\mbb{Z}^d,\\ \text{and}~\not\exists~\varnothing \subsetneq I \subseteq \{1, \dotsc, k\}~\text{such that}\\
\sum_{i \in I}\mbf{g}^{t_i} \equiv \mbf{0} \Mod \mbf{H}\mbb{Z}^d
\end{array}
\right\}.
\]
There are $\Delta-1$ non-zero cosets of $\mbb{Z}^d / \mbf{H}\mbb{Z}^d$, so if $(\mbf{g}^{t_i})_{i=1}^k \in \Omega$, then $k \le \Delta-1$. 
Thus, 
\(
|\Omega| \le \bar{f}(\Delta) := (\Delta-1)^{\Delta-1}.
\)

Given a sequence $\rho = (\mbf{g}^{t_i})_{i=1}^k\in \Omega$, consider a sequence $\tau = (v_i)_{i=1}^k$ of nodes in $V$.
Throughout the proof, we use the notation $\tau|\rho$ to denote a sequence $\omega = (\mbf{g}^{t_i})_{i=1}^k\in \Omega$ and a node sequence $\tau = (v_i)_{i=1}^k$ of the same length.  

For $\tau|\rho$, let $\mcf{T}_{\tau|\rho}$ denote the Minkowski sum of the $k$ convex hulls of uncapacitated unit flows in $H'$ from $(v_1, \mbf{0})$ to $(v_1,\mbf{g}^{i_1})$, from $(v_2,\mbf{g}^{i_1})$ to $(v_2,\mbf{g}^{i_1}+\mbf{g}^{i_2})$, and so on up to flow from $(v_k,\sum_{j=1}^{k-1}\mbf{g}^{i_j})$ to $(v_k,\sum_{j=1}^k\mbf{g}^{i_j})$:
\[
\mcf{T}_{\tau|\rho} = \left\{\mbf{x}\in\mathbb{R}^{A'}: 
\begin{array}{@{\hskip 0 cm}r@{\hskip .1 cm}c@{\hskip .1 cm}l@{\hskip 0 cm}}
\multicolumn{3}{@{\hskip 0 cm}l}{\vphantom{\mbf{f}_{v_j}^{k}}\exists~\mbf{x}_1, \dotsc, \mbf{x}_k \ge \mbf{0}~\text{such that}}\\
\mbf{D}'\mbf{x}_j &=& \mbf{f}_{v_j}^{k} ~\forall ~j \in \{1, \dotsc, k\},~\text{and}\\
\mbf{x}&=&\sum_{j=1}^k\mbf{x}_j
\end{array}\right\},
\]
where $\mbf{f}_{v_j}^{k} =\mbf{e}_{(v_j, \sum_{\ell=1}^{j-1}\mbf{g}^{i_{\ell}})} - \mbf{e}_{(v_j, \sum_{\ell=1}^j\mbf{g}^{i_{\ell}})} \in \mbb{R}^{V'}$.

The previous equation demonstrates that $\mcf{T}_{\tau|\rho}$ has an extended formulation in the space of variables $(\mbf{x}, \mbf{x}_1,\dots,\mbf{x}_k)$ that can be described by $k \cdot |A'|$ inequalities and $k\cdot|V'| + |A'|$ equations.
Using Balas's result on the union of polyhedra (\cite{B1985}), the convex hull $\mathcal{T}$ of the union of all $\mathcal{T}(\tau|\rho)$ can be written as:
\[
\begin{array}{r@{\hskip .1 cm}c@{\hskip .1 cm}l}
\mcf{T} &:=& \conv\left(\bigcup_{\tau|\rho} \mcf{T}_{\tau|\rho}\right)\\
& = & 
\left\{\mbf{x}\in\mathbb{R}^{A'}:
\begin{array}{@{\hskip 0 cm}r@{\hskip .1 cm}l@{\hskip .1 cm}c@{\hskip .1 cm}ll@{\hskip .1 cm}}
\multicolumn{5}{@{\hskip 0 cm}l}{\exists~\mbf{x}_{\tau|\rho}, \lambda_{\tau|\rho},~\text{and}~\mbf{x}_{j,\tau|\rho},}\\
\multicolumn{5}{@{\hskip 0 cm}l}{\forall~\tau|\rho~\text{and}~j \in \{1, \dotsc, k\},~\text{such that}}\\
\sum_{\tau|\rho}&\mbf{x}_{\tau|\rho} & = & \mbf{x},\\
\sum_{j=1}^k&\mbf{x}_{j,\tau|\rho}& = & \mbf{x}_{\tau|\rho} & \forall~\tau|\rho\\[.1cm]
&\mbf{D}'\mbf{x}_{j,\tau|\rho} &=& \lambda_{\tau|\rho}\mbf{f}_{v_j}^k & \forall~\tau|\rho,j\\[.1cm]
\sum_{\tau|\rho}&\lambda_{\tau|\rho} &=& 1,\\[.1cm]
\multicolumn{4}{l}{\text{and}~\mbf{x}_{j,\tau|\rho}\ge \mbf{0},~\lambda_{\tau|\rho}\ge0}& \forall~\tau|\rho,j\\[.1cm]
\end{array}
\right\}.
\end{array}
\]
The previous formulation has $1+k\cdot|A'|$ inequalities for each $\tau|\rho$.
%
There are at most $\bar{f}(\Delta)\cdot|V|^{\Delta-1}$ choices of $\tau|\rho$, so $\mcf{T}$ has an extended formulation of size at most
\[
\begin{array}{@{\hskip 0 cm}r@{\hskip .1 cm}l}
& \bar{f}(\Delta)|V|^{\Delta-1} (1+\Delta|A'|)
\in O(\Delta^{\Delta+1}|A|\cdot |V|^{\Delta-1}).
\end{array}
\]
Thus, it suffices to prove $\mathcal{S}$ is a projection of $\mathcal{T}$ via
\[
\textstyle
\pi(\mbf{x})_{(v,w)} := \sum_{i=0}^{\Delta-1}\mbf{x}_{((v,\mbf{g}^i), (w, (\mbf{g}^i+\mbf{W}_{(v,w)})\Mod \mbf{H}\mbb{Z}^d))}.
\]

For each $\tau|\rho$, we have 
\[
\rec(\mcf{S}) = \rec(\pi(\mcf{T}_{\tau|\rho})) = \left\{\mbf{y}\in\mathbb{R}^A_+:~\mbf{D}\mbf{y}=\mbf{0}\right\},
\]
where $\rec(\cdot)$ denotes the recession cone.
Also, $\rec(\pi(\mcf{T})) = \cap_{\tau|\rho}\rec(\pi(\mcf{T}_{\tau|\rho}) $, so $\rec(\mcf{S}) = \rec(\pi(\mcf{T}))$.
Hence, we only need to show that the vertices of $\pi(\mathcal{T})$ are contained in $\mcf{S}$ and the vertices of $\mathcal{S}$ are contained in $\pi(\mathcal{T})$.
    
Let $\mbf{y}$ be a vertex of $\pi(\mathcal{T})$. 
Hence, $\mbf{y}$ is the image of a vertex $\mbf{z}$ of some $\mcf{T}_{\tau|\rho}$. 
As $\mbf{z}$ is a vertex of $\mcf{T}_{\tau|\rho}$, we know $\mbf{z} = \sum_{j=1}^k\mbf{z}_j$, where each $\mbf{z}_j$ is a unit flow from $(v_j,\sum_{\ell=1}^{j-1}\mbf{g}^{i_{\ell}})$ to $(v_j,\sum_{\ell=1}^j\mbf{g}^{i_{\ell}})$.
If $\mbf{z}_j$ is not integer-valued, then it is the convex combination of other unit flows from $(v_j,\sum_{\ell=1}^{j-1}\mbf{g}^{i_{\ell}})$ to $(v_j,\sum_{\ell=1}^j\mbf{g}^{i_{\ell}})$; however, this contradicts that $\mbf{z}$ is a vertex.
Hence, $\mbf{y}:=\pi(\mbf{z})$ is the sum of $k$ nonnegative circulations in $H$, so it is a nonnegative circulation. 
Also, $\sum_{\ell=1}^k\mbf{g}^{i_{\ell}}\equiv \mbf{f} \Mod \mbf{H}\mbb{Z}^d$.
Thus, $\pi(\mathcal{T})\subseteq \mcf{S}$.

Let $\mbf{y}$ be a vertex of $\mcf{S}$. 
Hence, $\mbf{y}$ is a nonnegative integer circulation in $H$ with $\mbf{W}\mbf{y}\equiv \mbf{f}\Mod \mbf{H}\mbb{Z}^{d}$.
Also, $\mbf{y}$ can be decomposed into directed cycles in $H$, say $\mbf{y}=\sum_{i=1}^k\mbf{y}_i$, where each $\mbf{y}_i$ satisfies $\mbf{y}_i\ge \mbf{0}, \mbf{D}\mbf{y}_i=\mbf{0}$, and $\mbf{y}_i\in \{0,1\}^A$.
There is no zero-sum subsequence of $(\mbf{y}_i)_{i=1}^k$ because $\mbf{y}$ is a vertex of $\mcf{S}$, i.e., there is no $\varnothing \subsetneq I \subseteq \{1, \dotsc, k\}$ such that $\mbf{y}_0:=\sum_{i\in I}\mbf{y}_i$ satisfies $\mbf{W}\mbf{y}_0 \equiv \mbf{0} \Mod \mbf{H}\mbb{Z}^d$.
Otherwise, $\mbf{y} \pm \mbf{y}^0 \in \mcf{S}$ and $\mbf{y} = \frac12(\mbf{y} - \mbf{y}_0) + \frac12(\mbf{y}+\mbf{y}_0)$, implying that $\mbf{y}$ is not a vertex.
Therefore, $(\mbf{W}\mbf{y}_i\Mod \mbf{H}\mbb{Z}^{d})_{i=1}^k \in \Omega$, and $(\mbf{y}_i)_{i=1}^k$ is the projection under $\pi$ of $k$ paths from $(v_1, \mbf{0})$ to $(v_1,\mbf{g}^{i_1})$, from $(v_2,\mbf{g}^{i_1})$ to $(v_2,\mbf{g}^{i_1}+\mbf{g}^{i_2})$, and so on up to from $(v_k,\sum_{j=1}^{k-1}\mbf{g}^{i_j})$ to $(v_k,\sum_{j=1}^k\mbf{g}^{i_j})$, where each $v_j$ is in the cycle $\mbf{y}_j$.
Hence, $\mbf{y}=\sum_{j=1}^k\mbf{y}_j\in\pi(\mcf{T}_{\tau|\rho})$.
Thus, $\mathcal{S}\subseteq \pi(\mathcal{T})$.
\qed
\end{proof}

\section{A proof of Theorem~\ref{thmMain}.}\label{secProof}
%
Suppose $\mbf{A} \overline{\mbf{x}} = \mbf{b}$ for $\overline{\mbf{x}} \in \mbb{R}^n$.
By Lemma \ref{lem:rational_cone}, if $\overline{\mbf{x}}\in\mbb{Z}^n$, then $\conv(\mcf{P}(\mbf{A}, \mbf{b}) \cap \mbb{Z}^n)$ has an extended formulation of size in $O(n^2)$.
Thus, for the remaining proof, we assume $\overline{\mbf{x}}\notin\mbb{Z}^n$.

As $\mbf{A}$ is strictly $\Delta$-modular of full-rank, there exists $B \subseteq \{1, \dotsc, m\}$ be such that $|\det \mbf{A}_B| = \Delta$; $B$ can be found in polynomial time, e.g., using Gaussian elimination.
We can reorder the rows and columns of $\mbf{A}$ to assume that
\[
\mbf{A} = \begin{bmatrix}
    \mbf{A}_B \\
    \mbf{A}_N
\end{bmatrix} ~\text{with}~
\begin{bmatrix}
    \mbf{I}_n & \mbf{0}\\
    -\mbf{A}_N\mbf{A}_B^{-1} & \mbf{I}_{m-n}
\end{bmatrix}\mbf{A} =
\begin{bmatrix}
    \mbf{A_B}\\
    \mbf{0}
\end{bmatrix}.
\]

As $|\det \mbf{A}_B| = \Delta$, there are $\Delta-1$ nonzero cosets of $\mbb{Z}^n /\mbf{A}_B\mbb{Z}^{n}$; denote them by $\mbf{g}^i$ for $i \in \{1, \dotsc, \Delta-1\}$.
Given that $\mbf{b}_B = \mbf{A}_B \overline{\mbf{x}}\in\mbb{Z}^n$ and $\overline{\mbf{x}}\notin\mbb{Z}^n$, there is some $i^* \in \{1, \dotsc, \Delta-1\}$ such that $\mbf{b}_B\equiv \mbf{g}^{i^*} \Mod \mbf{A}_B \mbb{Z}^n$.

By Lemma \ref{lem:PQ_U}, $\mcf{Q}(\mbf{A}, \mbf{b})$ and $\conv(\mcf{P}(\mbf{A}, \mbf{b}) \cap \mbb{Z}^n)$ are affinely equivalent. 
Thus, it suffices to show $\mcf{Q}(\mbf{A}, \mbf{b})$ has a compact extended formulation.
Note that $\mbf{A}_n\mbf{A}_B^{-1}$ is TU because $\mbf{A}_B$ is a basis of the strictly $\Delta$-modular matrix $\mbf{A}$.
Thus, $[-\mbf{A}_n\mbf{A}_B^{-1}~\mbf{I}_{m-n}]$ is TU.
Apply Lemma~\ref{lem:reformulation_delta} with $\mbf{U} = [\mbf{I}_n~\mbf{0}]$, $\mbf{R} =  [-\mbf{A}_n\mbf{A}_B^{-1}~\mbf{I}_{m-n}]$, and $\mbf{H} = \mbf{A}_B$.
As $\mbf{b}\in\mathrm{span}(\mbf{A})$ and $\mbf{b}_B\equiv \mbf{g}^{i^*} \Mod \mbf{A}_B \mbb{Z}^n$, we have
\[
\mcf{Q}(\mbf{A}, \mbf{b}) = \conv\left\{\mbf{y}\in\mbb{Z}_{+}^m:
\begin{array}{l@{\hskip .1 cm}}
\begin{bmatrix}
    -\mbf{A}_N\mbf{A}_B^{-1}~ \mbf{I}_{m-n}
\end{bmatrix}\mbf{y}= \mbf{0}\\
\text{and}~\mbf{y}_B\equiv \mbf{g}^{i^*} \Mod \mbf{A}_B \mbb{Z}^n
\end{array}
\right\}.
\]

Let $M = M(\mbf A^\top)$.
The dual matroid $M^*$ represented by $\begin{bmatrix}
    -\mbf{A}_N\mbf{A}_B^{-1}~ \mbf{I}_{m-n}
\end{bmatrix}$ is graphic because $M$ is cographic.

We claim that there exists a node-arc incidence matrix $\mbf{D}\in\{0,\pm 1\}^{(m-n+1)\times m}$ such that
\[
\mcf{Q}(\mbf{A}, \mbf{b}) = \conv\left\{\mbf{y}\in\mbb{Z}_{+}^m:
\begin{array}{l}\mbf{D}\mbf{y}= \mbf{0}\\
\text{and}~\mbf{y}_B\equiv \mbf{g}^{i^*} \Mod \mbf{A}_B \mbb{Z}^n
\end{array}\right\}.
\]
Lemma~\ref{lem:compact_ef_delta} will then imply that $\mcf{Q}(\mbf{A}, \mbf{b})$ has a compact extended formulation of size in $O((m-n+1)^{\Delta-1} m)$.
We have $m\in O(n)$ because $M$ is cographic~\cite[pg. 145]{Oxley}.
Thus, $O((m-n+1)^{\Delta -1}m) \subseteq O(n^{\Delta})$.

To prove the claim, it suffices to argue that there exists such a $\mbf{D}$ for which, given $\mbf{y}\in\mathbb{R}^m$, we have $\mbf{D}\mbf{y}=\mbf{0}$ if and only if $\begin{bmatrix} -\mbf{A}_N\mbf{A}_B^{-1}~ \mbf{I}_{m-n}
\end{bmatrix}\mbf{y}=\mbf{0}$.
%

Let $H$ be a graph such that 
\[
M^*=M(H) = M([-\mbf{A}_N\mbf{A}_B^{-1}~ \mbf{I}_{m-n}]).
\]
We can assume $M(H) = M(\mbf{E})$, where
\[
\begin{bmatrix}
    \mbf{E}\\
    -\mbf{1}^\top \mbf{E}
\end{bmatrix}\in\{0,\pm 1\}^{(m-n+1)\times m}
\]
is the node-arc incidence matrix of $H$ with arbitrary orientation. 
The row sum of a node-arc incidence matrix is $\mbf{0}$, so the last row is $-\mbf{1}^\top \mbf{E}$.
The matrix $\mbf{E}$ can be found in polynomial time; see, e.g., \cite[page 326]{bixby1980converting}.
Suppose $\mbf{E}=[\mbf{E}_1~\mbf{E}_2]$, where $\mbf{E}_2\in\mbb{Z}^{(m-n)\times(m-n)}$. 
The matrix $\mbf{E}_2$ is invertible because the last $m-n$ elements of $M^*$ form a basis. 
Thus, $[\mbf{E}_2^{-1}\mbf{E}_1~\mbf{I}_{m-n}]$ is a TU representation of $M^*$. 
By Lemma~\ref{lemOx641}, $(\mbf{E}_2^{-1}\mbf{E}_1)^\# = (-\mbf{A}_N\mbf{A}_B^{-1})^\#$.
By Lemma~\ref{lemOx1017}, $\mbf{E}_2^{-1}\mbf{E}_1$ can be obtained from $-\mbf{A}_N\mbf{A}_B^{-1}$ by negating some rows and columns, say $\mbf{P}_1(-\mbf{A}_N\mbf{A}_B^{-1})\mbf{P}_2 = \mbf{E}_2^{-1}\mbf{E}_1$, where $\mbf{P}_1$ and $\mbf{P}_2$ are diagonal matrices with entries in $\{-1, 1\}$. 

Consider the node-arc incidence matrix
\[
\mbf{D}:=\begin{bmatrix}
    \mbf{E}\\
    -\mbf{1}^\top \mbf{E}
\end{bmatrix}
\left[\begin{array}{ll}
    \mbf{P}_2^{-1} & \mbf{0} \\
    \mbf{0} & \mbf{P}_1
\end{array}\right].
\]
For $\mbf{y}\in\mathbb{R}^m$, the equation $\mbf{D}\mbf{y}=\mbf{0}$ holds if and only if $\begin{bmatrix} -\mbf{A}_N\mbf{A}_B^{-1}~ \mbf{I}_{m-n}
\end{bmatrix}\mbf{y}=\mbf{0}$ because
\[
\begin{array}{rcl}
\mbf{E}
\left[\begin{array}{ll}
    \mbf{P}_2^{-1} & \mbf{0} \\
    \mbf{0} & \mbf{P}_1
\end{array}\right]
&=&\begin{bmatrix}
    \mbf{E}_1 \mbf{P}_2^{-1}~ \mbf{E}_2\mbf{P}_1
\end{bmatrix}\\
&=&\begin{bmatrix}
    \mbf{E}_2 \mbf{P}_1(-\mbf{A}_{N}\mbf{A}_B^{-1})~\quad \mbf{E}_2\mbf{P}_1
\end{bmatrix}\\[.1cm]
&=&
\mbf{E}_2\mbf{P}_1\begin{bmatrix}
    -\mbf{A}_N\mbf{A}_B^{-1}~ \mbf{I}_{m-n}
\end{bmatrix},
\end{array}
\]
and $\mbf{E}_2\mbf{P}_1$ is invertible.
\qed

\section{Necessity of Conditions \emph{\ref{Assume2}} and \emph{\ref{Assume3}}}\label{secExamples}

\subsection{Necessity of Condition \emph{\ref{Assume2}}}\label{secCevallos}

Cevallos et al.~\cite{CWZ2018} demonstrate that there exists a polyhedron $\mcf{P}(\mbf{A}, \mbf{b})$ such that $\mbf{A}$ is strictly bimodular and $\mbf{b} \in \Span(\mbf{A})$, yet $\conv(\mcf{P}(\mbf{A}, \mbf{b}) \cap \mbb{Z}^n)$ does not have a compact extended formulation; see Section 5.8 in~\cite{CWZ2018}.
In this subsection, we observe that their example satisfies all conditions in Theorem~\ref{thmMain} except \emph{\ref{Assume2}}; this highlights the necessity of \emph{\ref{Assume2}} for the statement to hold true.
Moreover, we show that $M(\mbf{A}^\top)$ is graphic.
Hence, we cannot replace \emph{\ref{Assume2}} in Theorem~\ref{thmMain} with the condition that $M(\mbf A^{\top})$ is graphic or cographic, or more generally, that $M(\mbf A^{\top})$ is regular.

Let $D = (V,A)$ be a complete directed graph on $n$ nodes, where $n$ is even, and let $\mbf{D}$ be the node-arc incidence matrix of $D$. 
Let $v_0\in V$ and define $\overline{\mbf{D}}$ as the matrix obtained by deleting the row of $\mbf{D}$ corresponding to $v_0$.
Set
\[
\mbf{A} := \left[\begin{array}{rrr}
        -\mbf{I}_A & \overline{\mbf{D}}{}^\top & \mbf{0}\\
        -\mbf{I}_A & \mbf{0}\phantom{^\top} & \mbf{0}\\
        \mbf{0}^\top & \mbf{1}^\top & -2\\
        \mbf{0}^\top & -\mbf{1}^\top & 2\\
    \end{array}\right].
\]
We consider $\overline{\mbf{D}}$ rather than $\mbf{D}$ so that $\mbf{A}$ has full column rank.
Consider the polyhedron
\[
\begin{array}{r@{\hskip .1 cm}c@{\hskip .1 cm}l}
    \mcf{P}&:=&\left\{\begin{bmatrix}
        \mbf{x} \\
        \mbf{y} \\
        z
    \end{bmatrix}\in\mbb{R}^A_{+}\times \mbb{R}^{V - v_0} \times\mbb{R}:~
    \begin{array}{@{\hskip .1 cm}l}
    \forall ~(v,w)\in A:\\[.1cm]
    \quad x_{(v,w)}\ge y_w - y_v\\[.1 cm]
    \sum_{v\in V} y_v = 2z + 1
    \end{array}\right\}\\[.25 cm]
    &=&\left\{\begin{bmatrix}
        \mbf{x} \\
        \mbf{y} \\
        z
    \end{bmatrix}\in\mbb{R}^A\times \mbb{R}^{V - v_0} \times\mbb{R}:
    \mbf{A}
    \begin{bmatrix}
        \mbf{x} \\
        \mbf{y} \\
        z
    \end{bmatrix} \le 
    \left[\begin{array}{r}
        \mbf{0}\\
        \mbf{0}\\
        1\\
        -1
    \end{array}\right]\right\},
\end{array}
\]

The matrix $\mbf{A}$ is strictly bimodular, and $\mbf{b}\in \Span(\mbf{A})$ because the column corresponding to $z$ is $(\mbf{0}_A,\mbf{0}_A, -2, 2)^\top = -2\mbf{b}$.
The matrix $\mbf{A}$ is not cographic because the cographic property is preserved under deleting rows and columns, yet $\overline{\mbf{D}}$ has too many distinct rows to be cographic. 
Indeed, $\overline{\mbf{D}}$ has rank $|V|-1$ and $O(|V|^2)$ distinct rows, but the number of distinct non-zero rows in a cographic matrix is linear in the rank; see~\cite{J1979}.
Cevallos et al.~\cite{CWZ2018} show $ \conv(\mcf{P} \cap (\mbb{Z}^A \times \mbb{Z}^{V-v_0} \times\mbb{Z}))$ does not have a compact extended formulation in Theorem 5.4.

We briefly argue why $M(\mbf{A}^\top)$ is graphic.
The class of graphic matroids is closed under parallel extensions and adding \textbf{coloops}, which is an element not in span of the other elements.
The last two rows of $\mbf A$ are parallel, so we may delete the last row of $\mbf A$; call the result $\mbf{A}_0$.
The last row of $\mbf{A}_0$ is a coloop of $M(\mbf{A}_0^\top)$, so we may delete it; call the resulting matrix $\mbf{A}_1$.
For each row of $\mbf{A}_1$ with two $-1$ entries, subtract the column with the first $-1$ entry from the column with the second $-1$ entry.
The matrix $\mbf{A}_2$ obtained from $\mbf{A}_1$ by performing this operation has at most two nonzero entries in each row, and each row with two nonzero entries has a $1$ and a $-1$.
Thus, after scaling each row by $-1$, the transpose of the resulting matrix $\mbf{A}_3$ is a column submatrix of the canonical representation of $M(K_{|A|+|V|})$.
Hence, $M(\mbf{A}_3^\top)$ is graphic.
Column operations and row scaling do not change the row matroid.
Hence, $M(\mbf{A}^\top)$ is also graphic.
Thus, condition \emph{\ref{Assume2}} cannot be extended to include matrices representing a graphic row matroid.

\subsection{Necessity of Condition \emph{\ref{Assume3}}}\label{sec:ola}

Jia et al.~\cite{jia2023exact} demonstrate that there exists a polyhedron $\mcf{P}(\mbf{A}, \mbf{b})$ such that $\mbf{A}$ is strictly bimodular and $\mbf{b} \not\in \Span(\mbf{A})$, but $\conv(\mcf{P}(\mbf{A}, \mbf{b}) \cap \mbb{Z}^n)$ does not have a compact extended formulation.
In this subsection, we observe that their example also satisfies condition \emph{\ref{Assume2}} in Theorem, highlighting the necessity of \emph{\ref{Assume3}} for the statement to hold true.

Let $G=K_{n,n}=(U\cup V, E)$ be the complete bipartite graph on $2n$ nodes.
Each edge is colored red or blue.
Let $\gamma\in\{0,1\}^E$ indicate the edge coloring with $\gamma_e=1$ if and only if $e\in E$ is red.
The parity bipartite perfect matching polytope ${\rm PBPM}(n)$ is the convex hull of characteristic vectors of perfect matchings with an odd number of red edges:
\[
\begin{array}{r@{\hskip .1 cm}l@{\hskip 0 cm}l}
& {\rm PBPM}(n)\\
    =&\conv\left\{\mbf{x}\in\mbb{Z}_+^E:\mbf{A}_G\mbf{x}_e = \mbf{1}=\frac1n\mbf{A}_G\mbf{1},\right. &\left.\gamma^\top \mbf{x} \equiv 1 \Mod 2\right\}\\[.15 cm]
    =&\conv\left\{\mbf{x}\in\mbb{Z}_+^E:\mbf{D}_G\mbf{x}_e =\frac1n\mbf{D}_G\mbf{1} ,\right. &\left.\gamma^\top \mbf{x} \equiv 1 \Mod 2\right\},
\end{array}
\]
where $\mbf{A}_G$ is the node-edge incidence matrix of $G$, and $\mbf{D}_G$ is the node-arc incidence matrix of the graph $G$ with the orientation of all edges from nodes in $U$ to nodes in $V$. 
As $G$ is bipartite, we know that the last equation is true by negating the rows corresponding to the nodes $V$.
Jia et al.~\cite{jia2023exact} show ${\rm PBPM}(n)$ does not have a compact extended formulation. 

By introducing an auxiliary integer variable $z$, ${\rm PBPM}(n)$ can be viewed as a projection of 
\[
\conv\left\{\begin{bmatrix}
    \mbf{x}\\
    z
\end{bmatrix}\in\mbb{Z}^E\times\mbb{Z}:
\left[
\begin{array}{@{\hskip 0 cm}r@{\hskip .1 cm}r}
\mbf{D}_G & \mbf{0}\\
-\mbf{D}_G & \mbf{0}\\
-\mbf{I}_E & \mbf{0}\\
\gamma^\top & -2 \\
-\gamma^\top & 2
\end{array}
\right]
\begin{bmatrix}
    \mbf{x}\\
    z
\end{bmatrix}
\le
\left[
\begin{array}{@{\hskip 0 cm}r}
\frac1n\cdot\mbf{D}_G\mbf{1}\\[.1 cm]
-\frac1n\cdot\mbf{D}_G\mbf{1}\\
\mbf{0}\\
1\\
-1
\end{array}
\right]
\right\}.
\]
The coefficient matrix $\mbf{A}$ is full-rank and strictly bimodular.
However, the constraints $\mbf{x}\ge \mbf{0}$ imply that $\mbf{b} \not \in \Span(\mbf{A})$.

We briefly argue why $M(\mbf{A}^\top)$ is cographic.
The class of cographic matroids is closed under parallel extensions and adding coloops.
The last two rows of $\mbf A$ are parallel, so we may delete the last row; call the resulting matrix $\mbf{A}_0$.
The last row of $\mbf{A}_0$ is a coloop of $M(\mbf{A}^\top_0)$, so we may delete it.
The resulting matrix is a parallel extension of $[\mbf{D}_G^\top~-\mbf{I}_E]$, which represents a cographic matroid. Hence, $M(\mbf{A}^\top)$ is cographic.

\section{A proof of Corollary~\ref{corConforti}}\label{secGeneralizing}
In the proof of Corollary~\ref{corConforti} provided by Conforti et al.~\cite{CFHW2022}, the authors use a result of Lov\'{a}sz about graphs without disjoint odd cycles (see Theorem \ref{thm:Lovasz}) together with an even-face embedding to derive a reformulation~\cite[Lemma 12]{CFHW2022}.
Our proof of Corollary~\ref{corConforti} uses the same result of Lov\'{a}sz, but we focus on the cographic matroid property to bypass even-face embeddings and derive a more general reformulation approach (Lemma \ref{lem:reformulation_delta}).

A matroid $M = (E, r)$ is {\bf 3-connected} if there is no partition $(X,Y)$ of $E$ such that $|X|,|Y| \ge j$ and $r(X) + r(Y) - r(E) < j$ for some $j \in \{1,2\}$.
A matroid $M = (E, r)$ is {\bf internally 4-connected} if it is 3-connected and $\min\{|X|, |Y|\}=3$ for every partition $(X,Y)$ of $E$ such that $|X|,|Y| \ge 3$ and $r(X) + r(Y) - r(E) < 3$.
Seymour~\cite{seymour1995matroid} proved that if $M$ is an internally 4-connected regular matroid, then $M$ is either graphic, cographic, or a specific $10$-element matroid called $R_{10}$.
The follow result of Lov\'asz refines this in an important special case.
\begin{theorem}\label{thm:Lovasz}({\bf Lov\'{a}sz, cited in \cite{seymour1995matroid} as Theorem 6.7, page 546-547})
Let $M$ be the column matroid of the node-edge incidence matrix of a graph $G$ with $\mathrm{ocp}(G)\le 1$.
Then, $M$ is regular.
If $M$ is also internally 4-connected, then
\begin{enumerate}[label=(\arabic*), leftmargin = *]
\item if $M\cong R_{10}$, then $G\cong K_5$;
\item if $M$ is graphic, then there exists a node $v$ such that $G-v$ is bipartite, or there exists a circuit $\{e,f,g\}$ such that $G-\{e,f,g\}$ is bipartite;
\item if $M$ is cographic, then $G$ has an even-face embedding in the projective plane.
\end{enumerate}
\end{theorem}
We can then have the following corollary.
\begin{corollary}
Assume that $G = (V,E)$ is 4-connected with $\mathrm{oct}(G)\ge 4$ and $\mathrm{ocp}(G)=1$.
Let $\mbf{A}^\top$ be the node-edge incidence matrix of $G$.
Then $M = M(\mbf{A}^\top)$ is cographic.
\end{corollary}
\begin{proof}
By Theorem \ref{thm:Lovasz} we know that $M$ is regular.
It is straightforward, but technical, to show that $M$ is internally 4-connected (in fact, we show that $M$ is 4-connected).

\begin{claim}
$M$ is internally 4-connected.
\end{claim}

\begin{cpf}
Let $r$ be the rank function of $M$.
For a subgraph $H$ of $G$ we write $c(H)$ and $b(H)$ for the number of components and bipartite components, respectively, of $H$.
The matroid $M$ is a signed-graphic matroid with underlying graph $G$; we direct the reader to \cite{Oxley} for background information on signed-graphic matroids.
For this proof we only need the fact that if $X$ is a set of edges of $G$, then $r(X) = |V(G[X])| - b(G[X])$, where $G[X]$ is the graph induced by $X$; see \cite{SLILATY2005207}. 

Suppose $M$ is not internally 4-connected.
Hence, there is a partition $(X,Y)$ of $E$ such that 
$
r(X) + r(Y) - r(E) < j
$
and $|X|,|Y| \ge j$ for some $j \in \{1,2,3\}$.
Choose such a partition with $c(G[X]) + c(G[Y])$ as small as possible.
Let $Z = V(G[X]) \cap V(G[Y])$.
We will argue that 
\begin{equation}\label{eqC3Claim}
\begin{array}{l}
\text{each component of $G[X]$ or $G[Y]$ has a vertex that is not in $Z$.}
\end{array}
\end{equation}
By symmetry, it suffices to prove this for $G[X]$.
Note that each component of $G[X]$ either has vertex set contained in $Z$ and at most three vertices, or contains at least four vertices in $Z$, because $G$ is $4$-connected.
The graph $G$ has no triangles because $\oct(G) \ge 4$ and $\ocp(G) = 1$. 
Hence, each component of $G[X]$ with at most three vertices and vertex set contained in $Z$ is a path with one or two edges.

As a first step towards proving~\eqref{eqC3Claim}, assume for the sake of contradiction that every component of $G[X]$ has its vertex set contained in $Z$ and at most three vertices.
Under this assumption, we have $r(X) = |X|$.
Also, the vertex set of $G[Y]$ is equal to the vertex set of $G$.
Let $X' \subseteq X$ be such that $|X - X'| = j$ (possibly $X' = \varnothing$).
Note that $r(X-X')\le |X-X'| = |X| - |X'| = r(X) - |X'|$. 
Also, $r(Y \cup X') \le r(Y) + |X'|$ because adding $|X'|$ elements increases the rank by at most $|X'|$.
Thus,
\begin{equation}\label{eqC3E1}
\begin{array}{rl}
& r(X - X') + r(Y \cup X') - r(E)\\
\le& r(X) - |X'| + r(Y) + |X'| - r(E) < j,
\end{array}
\end{equation}
and $|X - X'|, |Y\cup X'| \ge j$.
If $r(Y \cup X') = r(E)$, then 
\[
r(X - X') + r(Y \cup X') - r(E) = |X - X'| \ge j.
\]
However, this contradicts~\eqref{eqC3E1}. 
Hence, $r(Y \cup X') < r(E)$.
Since $G[Y \cup X']$ and $G$ have the same vertex set and $G$ is connected, this means that $G[Y \cup X']$ has more components than $G$.
Moreover, because $|X - X'| = j \le 3$, this implies that $X - X'$ is a set of at most three edges whose deletion disconnects $G$.
However, then $G$ has a set of at most three vertices whose deletion disconnects $G$, which contradicts that $G$ is 4-connected.
Therefore, $G[X]$ has a component whose vertex set is not completely contained in $Z$, and contains at least four vertices in $Z$.
Moreover, this component has at least five vertices, and therefore has at least four edges.
Denote this edge set by $X^*$.
Thus, $|X^*| \ge 4$.

As a final step towards proving~\eqref{eqC3Claim}, assume to the contrary that $G[X]$ has a component with vertex set contained in $Z$ and at most three vertices.
Let $X''$ be the set of edges of this component, and recall that $|X''| \in \{1,2\}$.
The previous paragraph demonstrates that $G[X^*]$ is not this component.
Thus, $|X - X''| \ge |X^*| \ge 4$.
Note that $|V(G[X - X''])| = |V(G[X])| - (|X''| + 1)$ because $G[X'']$ is a path, and that $b(G[X - X'']) = b(G[X]) - 1$ because $G[X'']$ is a bipartite component of $G[X]$.
Therefore, $r(X - X'') = r(X) - |X''|$.
Consider the partition $(X - X'', Y \cup X'')$ of $E$.
We have
\[
\begin{array}{rl}
&r(X - X'') + r(Y \cup X'') - r(E)\\
\le& r(X) - |X''| + r(Y) + |X''| - r(E) < j.
\end{array}
\]
However, this contradicts the minimality of $c(G[X]) + c(G[Y])$ because $|X - X''| \ge 4 > j$.
Therefore, each component of $X$ has a vertex that is not in $Z$, which completes the proof of the statement claimed in~\eqref{eqC3Claim}.

Claim~\eqref{eqC3Claim} implies that $|Z| \ge 4 \cdot c(G[X])$ because $G$ is $4$-connected.
By applying the same reasoning to $G[Y]$, we have $|Z| \ge 4\cdot c(G[Y])$.
Without loss of generality, we may assume that $c(G[X]) \ge c(G[Y])$.
The graph $G$ has an odd cycle because $\ocp(G) = 1$. 
Thus, $r(E) = |V(G)|$ and 
\[
r(E) = |V(G)| = |V(G[X])| + |V(G[Y])| - |Z|.
\]
Given that $j \le 3$, we have 
\[
\begin{array}{r@{\hskip .2 cm}c@{\hskip .2 cm}l}
    2 &\ge& r(X) + r(Y) - r(E) \\
    &=& |Z| - b(G[X]) - b(G[Y]) \\
    & \ge& 4c(G[X]) - b(G[X]) - b(G[Y]) \\
    & \ge& (3c(G[X]) - b(G[X]) ) +  c(G[Y]) - b(G[Y])) \\
    & \ge& 2 + 0,
\end{array}
\]
which implies equality throughout.
Hence, $c(G[X]) = c(G[Y]) = b(G[X]) = b(G[Y]) = 1$ and $|Z| = 4$.
Therefore, $G[X]$ and $G[Y]$ are connected bipartite graphs that have four common vertices.
Thus, every odd cycle of $G$ contains two vertices in $Z$, which implies that $\oct(G) \le 3$, a contradiction.
Consequently, $M$ is internally 4-connected. 
\end{cpf}

Following the claim, we see that outcome \emph{(1)}, \emph{(2)}, or \emph{(3)} from Theorem~\ref{thm:Lovasz} holds. 
Outcomes \emph{(1)} and \emph{(2)} cannot hold because $\ocp(G) \ge 4$, so outcome \emph{(3)} holds, as desired.
\qed\end{proof}

We use the following lemma to recover Corollary \ref{corConforti}.
\begin{lemma}\label{lemLem8}({\bf Lemma 8 in \cite{CFHW2022}})
For a graph $G$ on $n$ nodes with edge-node incidence matrix $\mbf{A}$, and $\mbf{b}=\mbf{1}$, it holds that
\(
\STAB(G) = \conv(\mcf{P}(\mbf{A}, \mbf{b}) \cap \mbb{Z}^n)\cap [0,1]^n.
\)
\end{lemma}

Consider $\mbf{A}$ and $\mbf{b}$ in Lemma~\ref{lemLem8}.
Observe that $\mbf{b}=\mbf{1}$ is in the span of the columns of $\mbf{A}$ because it is obtained by summing all the columns and dividing by 2.
%
If $\ocp(G) = 1$, then $\mbf{A}$ is bimodular.
By Theorem \ref{thmMain} and Remark~\ref{remark2}, we know $\conv(\mcf{P}(\mbf{A}, \mbf{b}) \cap \mbb{Z}^n)$ admits an extended formulation of size in $O(n^2)$.
With Lemma~\ref{lemLem8}, we obtain an extended formulation and recover Corollary \ref{corConforti}.
We emphasize that the difference between this approach and Conforti et al. is the method of finding an extended formulation for $ \conv(\mcf{P}(\mbf{A}, \mbf{b}) \cap \mbb{Z}^n)$.

\section{Extended formulations for a new family}\label{secNew}

We conclude by showing Theorem~\ref{thmMain} is strictly more general than Corollary~\ref{corConforti}.
At a high level, if $\mbf{A}$ is covered by Corollary~\ref{corConforti}, then $M^*(\mbf{A}^\top)$ represents a projective planar graph $H$; Theorem~\ref{thmMain} is not restricted to this class of matrices.

Let $\mbf{A}^\top$ be the node-edge incidence matrix of a graph $G$ satisfying the conditions of Corollary~\ref{corConforti}.
Hence, $M(\mbf{A}^\top)$ represents the signed graphic matroid of the signed graph $\Sigma$ of $G$.
As $M(\mbf A^\top)$ is cographic, there is a (2-connected) graph $H$ such that $M^*(H) = M(\Sigma)$. 
Slilaty proved that $H$ is projective planar; see~\cite[Theorem 3, part 1]{SLILATY2005207}.

However, not all graphs are projective planar.
For example, let $\mbf {A}^\top$ be a strictly $\Delta$-modular representation of the cographic matroid of a complete graph $K_r$ with $r \ge 7$.
Let $[\mbf{I}_6 ~\mbf{D}_7]$ represent the graphic matroid $K_7$.
For
$
\mbf{A}^\top = \mbf{B}[\begin{array}{rr}
        -\mbf{D}_7^\top & \mbf{I}_{15}
    \end{array}],
$
where $|\det\mbf{B}| = \Delta$, we have $M(\mbf{A}^\top) = M([-\mbf{D}_7^\top ~ \mbf{I}_{15}]) = M^*([\mbf{I}_6 ~\mbf{D}_7])=M^*(K_7)$.
By Theorem~\ref{thmMain}, there is a compact extended formulation for $\conv(\mcf{P}(\mbf{A}, \mbf{b}) \cap \mbb{Z}^n)$ for suitable $\mbf{b}$.

\smallskip

\noindent{\bf Acknowledgements.} J. Paat was supported by a Natural Sciences and Engineering Research Council of Canada Discovery Grant [RGPIN-2021-02475]. 

\bibliographystyle{splncs04}
\bibliography{references.bib}

\begin{thebibliography}{10}
\providecommand{\url}[1]{\texttt{#1}}
\providecommand{\urlprefix}{URL }
\providecommand{\doi}[1]{https://doi.org/#1}

\bibitem{AFJKSWY2024}
Aprile, M., Fiorini, S., Joret, G., Kober, S., Seweryn, M., Weltge, S.,
  Yuditsky, Y.: Integer programs with nearly totally unimodular matrices: the
  cographic case. Available online at arXiv:2407.09477  (2024)

\bibitem{AEGOVW2016}
Artmann, S., Eisenbrand, F., Glanzer, C., Oertel, T., Vempala, S., Weismantel,
  R.: A note on non-degenerate integer programs with small sub-determinants.
  Operations Research Letters  \textbf{44}(5),  635--639 (2016)

\bibitem{AWZ2017}
Artmann, S., Weismantel, R., Zenklusen, R.: A strongly polynomial algorithm for
  bimodular integer linear programming. In: Proceedings of the 49th Annual ACM
  SIGACT Symposium on Theory of Computing. pp. 1206--1219 (2017)

\bibitem{AS2022}
Averkov, G., Schymura, M.: {On the maximal number of columns of a
  $\Delta$-modular matrix}. In: Aardal, K., L.~Sanit{\`a}, L. (eds.) {Integer
  Programming and Combinatorial Optimization}. pp. 29--42. Springer
  International Publishing, Cham (2022)

\bibitem{B1985}
Balas, E.: Disjunctive programming and a hierarchy of relaxations for discrete
  optimization problems. {SIAM Journal on Matrix Analysis and Applications}
  \textbf{6}(3),  466--486 (1985)

\bibitem{bixby1980converting}
Bixby, R.E., Cunningham, W.H.: Converting linear programs to network problems.
  Mathematics of Operations Research  \textbf{5}(3),  321--357 (1980)

\bibitem{C1963}
Camion, P.: Caract\'{e}risation des matrices unimodulaires. {Cahiers Centre
  \'{E}tudes Rech. Op\'{e}r.}  \textbf{5},  181--190 (1963)

\bibitem{CWZ2018}
Cevallos, A., Weltge, S., Zenklusen, R.: Lifting linear extension complexity
  bounds to the mixed-integer setting. Proceedings of the 2018 Annual ACM-SIAM
  Symposium on Discrete Algorithms pp. 788 -- 807 (2018)

\bibitem{CFHSW2020}
Conforti, M., Fiorini, S., Huynh, T., Joret, G., Weltge, S.: The stable set
  problem in graphs with bounded genus and bounded odd cycle packing number.
  In: Proceedings of the Thirty-First Annual ACM-SIAM Symposium on Discrete
  Algorithms (SODA). p. 2896–2915. Society for Industrial and Applied
  Mathematics, USA (2020)

\bibitem{CFHW2022}
Conforti, M., Fiorini, S., Huynh, T., Weltge, S.: Extended formulations for
  stable set polytopes of graphs without two disjoint odd cycles. Mathematical
  Programming  \textbf{192}(1),  547--566 (2022)

\bibitem{FiJoWeYu2022}
Fiorini, S., Joret, G., Weltge, S., Yuditsky, Y.: Integer programs with bounded
  subdeterminants and two nonzeros per row. In: 2021 IEEE 62nd Annual Symposium
  on Foundations of Computer Science (FOCS). pp. 13--24 (2022)

\bibitem{GWZ2018}
Glanzer, C., Weismantel, R., Zenklusen, R.: {On the number of distinct rows of
  a matrix with bounded subdeterminants}. {SIAM Journal on Discrete
  Mathematics}  \textbf{32}(3),  1706--1720 (2018)

\bibitem{gribanov2022delta}
Gribanov, D., Shumilov, I., Malyshev, D., Pardalos, P.: On {$\Delta$}-modular
  integer linear problems in the canonical form and equivalent problems.
  Journal of Global Optimization pp. 1--61 (2022)

\bibitem{GKS1995}
Grossman, J., Kulkarni, D., Schochetman, I.: {On the minors of an incidence
  matrix and its Smith normal form}. Linear Algebra and its Applications
  \textbf{218},  213--224 (1995)

\bibitem{J1979}
Jaeger, F.: Flows and generalized coloring theorems in graphs. Journal of
  Combinatorial Theory, Series B  \textbf{26}(2),  205--216 (1979)

\bibitem{jia2023exact}
Jia, X., Svensson, O., Yuan, W.: The exact bipartite matching polytope has
  exponential extension complexity. In: Proceedings of the 2023 Annual ACM-SIAM
  Symposium on Discrete Algorithms (SODA). pp. 1635--1654. SIAM (2023)

\bibitem{LPSX2022}
Lee, J., Paat, J., Stallknecht, I., Xu, L.: {Polynomial upper bounds on the
  number of differing columns of $\Delta$-modular integer programs}.
  Mathematics of Operations Research  (2022)

\bibitem{NNSZ2023}
N{\"a}gele, M., N{\"o}bel, C., Santiago, R., Zenklusen, R.: {Advances on
  Strictly $\Delta$-Modular IPs}. In: Del~Pia, A., Kaibel, V. (eds.) Integer
  Programming and Combinatorial Optimization. pp. 393--407. Springer
  International Publishing (2023)

\bibitem{NaSaZe2022}
Nägele, M., Santiago, R., Zenklusen, R.: Congruency-constrained {TU} problems
  beyond the bimodular case. In: Proceedings of the 2022 Annual ACM-SIAM
  Symposium on Discrete Algorithms (SODA). pp. 2743--2790 (2022)

\bibitem{Oxley}
Oxley, J.: Matroid Theory. Oxford University Press, 2 edn. (Feb 2011)

\bibitem{OW2021}
Oxley, J., Walsh, Z.: 2-modular matrices. SIAM Journal on Discrete Mathematics
  \textbf{36}(2),  1231--1248 (2022)

\bibitem{PSWX2023}
Paat, J., Stallknecht, I., Walsh, Z., Xu, L.: On the column number and
  forbidden submatrices for {$\Delta$}-modular matrices. {SIAM Journal on
  Discrete Mathematics}  \textbf{38}(1) (2024)

\bibitem{AS1986}
Schrijver, A.: Theory of linear and integer programming. John Wiley \& Sons,
  Inc. New York, NY (1986)

\bibitem{seymour1995matroid}
Seymour, P.: Matroid minors, Handbook of Combinatorics, vol.~1. North-Holland
  (Elsevier), Amsterdam (1995)

\bibitem{S1980}
Seymour, P.: Decomposition of regular matroids. Journal of Combinatorial Theory
  B  \textbf{28},  305--359 (1980)

\bibitem{SLILATY2005207}
Slilaty, D.C.: On cographic matroids and signed-graphic matroids. Discrete
  Mathematics  \textbf{301}(2),  207--217 (2005)

\bibitem{VC2009}
Veselov, S., Chirkov, A.: Integer programming with bimodular matrix. Discrete
  Optimization  \textbf{6},  220--222 (2009)

\end{thebibliography}

\end{document}